\documentclass[12pt]{amsart}

\usepackage{amsrefs}
\usepackage{amssymb}
\usepackage{fancyhdr}
\usepackage{bbm}
\usepackage{xcolor}
\usepackage{hyperref}
\hypersetup{
  colorlinks=true,
  linkcolor=blue,
  citecolor=blue
}

\newtheorem{theorem}{Theorem}
\newtheorem*{theorem*}{Theorem}
\newtheorem{lemma}[theorem]{Lemma}
\newtheorem{proposition}[theorem]{Proposition}
\newtheorem{claim}[theorem]{Claim}
\newtheorem{corollary}[theorem]{Corollary}

\newtheorem{question}[theorem]{Question}
\newtheorem{problem}[theorem]{Problem}

\newtheorem{remark}[theorem]{Remark}

\newtheorem{maintheorem}{Theorem}

\theoremstyle{definition}
\newtheorem{definition}[theorem]{Definition}
\newtheorem*{definition*}{Definition}
\newtheorem*{lemma*}{Lemma}
\usepackage{graphicx}
\usepackage{pstricks, enumerate, pst-node, pst-text, pst-plot}

\numberwithin{equation}{section}
\numberwithin{theorem}{section}

\newcommand{\N}{\mathbb{N}}

\newcommand{\Z}{\mathbb{Z}}

\usepackage{xparse}
\DeclareDocumentCommand\Pr{ m g }{\ensuremath{
    {   \IfNoValueTF {#2}
      {\mathbb{P}\left[{#1}\right]}
      {\mathbb{P}\left[{#1}\middle\vert{#2}\right]}%
    }
}}
\DeclareDocumentCommand\E{ m g }{\ensuremath{
    {   \IfNoValueTF {#2}
      {\mathbb{E}\left[{#1}\right]}
      {\mathbb{E}\left[{#1}\middle\vert{#2}\right]}%
    }
}}

\DeclareMathOperator{\supp}{supp}

\def\cc{\curvearrowright}

\def\ee{\mathrm{e}}

\def\cC{\mathfrak{C}}
\begin{document}

\title[]{A quantitative Neumann lemma for finitely generated groups}

\author[]{Elia Gorokhovsky}

\author[]{Nicol\'as Matte Bon}
\address{California Institute of Technology}
\address{
	CNRS,
	Institut Camille Jordan (ICJ, UMR CNRS 5208),
	Universit\'e Lyon 1}

\author[]{Omer Tamuz}



\thanks{Nicol\'as Matte Bon is partially supported by Labex MILyon and by ANR Gromeov (ANR-19-CE40-0007). Omer Tamuz was supported by a Sloan fellowship, a BSF award (\#2018397) and a National Science Foundation CAREER award (DMS-1944153)}

\date{\today}

\begin{abstract}
We study the coset covering function $\cC(r)$ of a finitely generated group: the number of cosets of infinite index subgroups needed to cover the ball of radius $r$. We show that $\cC(r)$ is of order at least $\sqrt{r}$ for all groups. Moreover, we show that $\cC(r)$ is linear for a class of amenable groups including virtually nilpotent and polycyclic groups, and that it is exponential for property (T) groups.
\end{abstract}

\maketitle
\section{Introduction}

Let $G$ be an infinite discrete group generated by a finite symmetric set $S$. We say that a coset $C = H g$ is a coset of infinite index (c.i.i.) of $G$ if $H$ is an infinite index subgroup of $G$. Neumann's Lemma~\cite{neumann1954groups} states that $G$ cannot be covered by a finite number of c.i.i.s. We refine this question, and ask: how many c.i.i.s are needed to cover $B_r(G,S)$, the ball of radius $r$ in $G$? For each $r$, denote by $\cC(r)$ the smallest number of c.i.i.s needed to cover $B_r$. That is,
\begin{align*}
    \cC_{G,S}(r) = \min \left\{ N\,:\, \exists\text{ c.i.i.s } C_1,\ldots,C_N \text{ s.t.\ } B_r(G,S) \subseteq \cup_{i=1}^N C_i\right\}.
\end{align*}
We call $\cC_{G,S}$ the {\em coset covering function} of $(G,S)$. 

A number of natural questions  arise: What lower and upper bounds on the coset covering function apply to all finitely generated groups? And how can these bounds be improved for groups with particular algebraic or geometric properties? The goal of this note is to establish some results in this direction, as well as to advertise the study of this invariant. 

 Neumann's Lemma suggests (but does not immediately imply) that the coset covering function should tend to infinity with $r$ for any infinte group $G$. Our first result confirms that this intuition, and moreover provides a quantitative lower bound that applies to all groups.
\begin{maintheorem}
\label{prop:quant-neumann}
For any infinite, finitely generated group $(G,S)$, we have $\cC_{G,S}(r) \geq \sqrt{r}/(4|S|)$.
\end{maintheorem}
The proof of this claim is a probabilistic argument that relies on the analysis of a random walk on the group, and follows from a result of Lyons~\cite{lyons2005asymptotic}. Using the same argument, the universal lower bound in Theorem \ref{prop:quant-neumann} can be improved when the distance from the origin of the random walk satisfies some upper bound with positive probability; see Theorem \ref{t-speed}.

Consider the smallest infinite group  $G=\Z$ and the standard generating set. Here $\cC_{G,S}(r) = 2r+1$, since the only infinite index subgroup is the trivial one.    Our next result shows that the coset covering function grows linearly for a class of amenable groups, namely all groups that admit a ``cautious random walk'' in the terminolgy of Erschler and Zheng \cite{erschler2021drift}. We recall the definition  below and mention here that, by  results of Tessera \cite{tessera2011asymptotic,tessera2013isoperimetric} and Erschler--Zheng  \cite{erschler2020isoperimetric}  this class includes  virtually nilpotent and polycyclic groups, solvable Baumslag-Solitar groups, lamplighter groups $F \wr \Z$ where $F$ is a finite group, solvable groups of finite Pr\"ufer rank and various lattices in solvable algebraic $p$-adic groups.
\begin{maintheorem}
\label{thm:poly}
For any infinite, finitely generated,  group $(G,S)$ that admits a cautious random walk there exist constants $c_1,c_2 > 0$ such that
\begin{align*}
    c_1 r \leq \cC_{G,S}(r) \leq c_2 r.
\end{align*}
\end{maintheorem}
The upper bound in the previous theorem is a consequence of the deep fact that a group which admits a cautious random walk must admit a finite index subgroup which surjects onto $\Z$. This follows from Shalom's property $\mathsf{H_{FD}}$ \cite{shalom2004harmonic}  established for such groups by  Erschler--Ozawa \cite{erschler2018finite-dimensional}. It is indeed  easy to see that $\cC_{G,S}(r)\le cr$ for any group which virtually surjects onto $\Z$  (see Claims \ref{lem:finite_index_bounds}-\ref{lem:quotient_upper_bound}). Note however that in all the special cases mentioned above, the existence of a virtual surjection to $\Z$ is obvious without appealing to Shalom's property. The lower bound follows instead from Theorem \ref{t-speed}.  For virtually nilpotent groups, we provide an alternative proof in \S \ref{sec:vn} which does not rely on random walks.


Finally, we note that the function $\cC_{G, S}(r)$ is always bounded above by an exponential function, since the ball $B_r(G,S)$ can be always covered by points, i.e. cosets of the trivial subgroups. Our last result shows that for some groups, the function $\cC_{G, S}(r)$ may actually grow exponentially.

\begin{maintheorem}
\label{thm:t}
For any infinite, finitely generated group $(G,S)$ with property (T) there exists a constant $\varepsilon$ such that
\begin{align*}
    \cC_{G,S}(r) \geq \ee^{\varepsilon r}.
\end{align*}
\end{maintheorem}
The constant $\varepsilon$ can be bounded from below by a monotone transformation of the Kazhdan constant of $(G, S)$.


\subsection{Open questions} We conclude this introduction by suggesting some problems for further investigation. 

We are not aware of any examples in which the coset covering function achieves the bound of Theorem~\ref{prop:quant-neumann}, or indeed of any in which it is sub-linear. Thus the following natural question is open:
\begin{question} \label{q-sublinear}
Does there exist an infinite, finitely generated group $(G,S)$ such that $\liminf_r \frac{1}{r}\cC_{G,S}(r)=0$?
\end{question}
Note that by Claim~\ref{lem:quotient_upper_bound} below, it is equivalent to consider free groups, because if any group with $d$ generators has a small covering function, then so does the free group with $d$ generators. Moreover, by Claim~\ref{lem:finite_index_bounds}, it suffices to consider the free group with two generators, since it contains every finitely generated free group with more generators as a finite index subgroup.

One difficulty in estimating the coset covering function is to control the cosets of different subgroups at the same time. However in some situations the cardinality of a minimal cover of $B_r(G, S)$ by c.i.i.s can be achieved (up to constants) by considering only covers by cosets of a single subgroup infinite index (s.i.i.): this is the case for the groups covered by Theorems \ref{thm:poly} and \ref{thm:t}. It is natural to wonder whether this is a general fact. To formalise this, let us introduce the following invariant
\begin{align*}
    \mathfrak{D}_{G,S}(r) = \min \left\{ N\,:\, \exists H \text{ s.i.i}
   \text{ and } g_1,\ldots,g_N \text{ s.t.\ } B_r(G,S) \subseteq \cup_{i=1}^N g_i H \right\}.
\end{align*}
Note that $\mathfrak{D}_{G, S}(r)$ has a  geometric interpretation: it is the  cardinality of the smallest ball of radius $r$ in the Schreier graph of $G/H$, when $H$ varies among all s.i.i.s of G (in contrast, $\cC_{G, S}$ lacks such an interpretation). In particular $\mathfrak{D}_{G, S}(r)\ge r$ trivially. We clearly have 
\[\cC_{G, S}(r)\le \mathfrak{D}_{G, S}(r)\]
so that we are lead to ask if the previous inequality can be strict asymptotically.
\begin{question}\label{q-single-subgroup}
Does there exists an infinte, finitely generated group $(G,S)$ such that $\liminf_{r\to \infty} \frac{\cC_{G, S}(r)}{\mathfrak{D}_{G, S}(cr)}=+\infty$ for every $c>0$?
\end{question}
Note that any example $(G, S)$ providing an affirmative answer to Question \ref{q-sublinear} would necessarily give  an affirmative answer to Question \ref{q-single-subgroup}.

Finally we propose the following concrete problem as a test case for the above questions. 

\begin{problem}
Estimate the coset covering function of the Grigorchuk group. 
\end{problem}

\subsection*{Acknowledgments} 
We thank Russ Lyons for introducing us to useful recent results in random walks.

\section{Preliminaries}

To simplify notation, define the equivalence relation $\sim$ on the space of functions $\N \to \N$ by $f \sim g$ if and only if there exist constants $c_1, c_2 > 0$ such that 
\begin{align*}
  c_1 f(c_1 r) \leq g(r) \leq c_2 f(c_2 r) \quad \text{ for all } r.
\end{align*}
A standard argument shows that if $S$ and $T$ generate $G$ then $\cC_{G,S} \sim \cC_{G,T}$. Accordingly, when we consider $\cC_{G,S}$ up to equivalence, we will write $\cC_G$ and omit the generating set.

\subsection{The covering function of finite index subgroups and of quotients}

\begin{claim}\label{lem:finite_index_bounds}
If $[G : H] < \infty$, then $\cC_G \sim \cC_H$. 

\end{claim}

\begin{proof}

Let $S$ be a generating set for $H$ and $T \supseteq S$ a generating set for $G$. We will show that
\begin{align}
\label{eq:finite-index}
\cC_{H, S}(r) \leq \cC_{G, T}(r) \leq [G : H]\cC_{H, S}(Cr + D),
\end{align}
for some $C, D>0$.
Fix $r$. Let $\ell = [G : H]$ and let $H g_1, \ldots, H g_\ell$ be the cosets of $H$ in $G$. 

To prove the first inequality in \eqref{eq:finite-index} we will construct a covering of $B_r(H, S)$ by $\cC_{G, T}(r)$ c.i.i.s.

Denote by $|\cdot|_S$ and $|\cdot|_T$ word lengths with respect to $S$ in $G$ and $T$ in $H$, respectively. Note that for any $h \in H$, we have $|h|_T \leq |h|_S$, since $S \subseteq T$. Therefore $B_r(H, S) \subseteq B_r(G, T)$.
Suppose $B_r(G, T) \subseteq \bigcup_{i=1}^m K_i g_i$, where each $K_i g_i$ is a c.i.i.\@ in $G$. Then \begin{align*}
B_r(H, S) &\subseteq B_r(G, T) \cap H \subseteq H \cap \bigcup_{i=1}^m K_i g_i 
= \bigcup_{i=1}^m (H \cap K_i) h_i
\end{align*}
for some $h_1,\ldots,h_\ell \in H$. Moreover, $[H : H\cap K_i] = \infty$, and each $(H \cap K_i) h_i$ is a c.i.i, and $\cC_{H, S}(r) \leq \cC_{G, T}(r)$.

We now turn to the second inequality in \eqref{eq:finite-index}. Since the inclusion of $(H, S)$ inside $(G, T)$ is a quasi-isometry, we can fix a constant $C>0$ such that $H \cap B_r(G, T) \subseteq B_{Cr+C}(H, S)$. We will construct a covering of $B_r(G, T)$ by $[G : H]\cC_{H, S}(Cr + D)$ c.i.i.s, where $D=C(k+1)$ and  $k = \max_{1 \leq i \leq \ell} |g_i|_T$. 

Let $m = \cC_{H, S}(Cr + D)$, and let  $B_{Cr + D}(H, S) \subseteq \bigcup_{i=1}^m K_i h_i$, with each $ K_i h_i$ a c.i.i.\ in $H$. Then \begin{align*}
    B_r(G, T) &= G \cap B_r(G, T)\\ 
    &= \left(\bigcup_{i=1}^\ell H g_i\right) \cap B_r(G, T) \\
    &= \bigcup_{i=1}^\ell (H \cap B_r(G, T) g_i^{-1}) g_i.
\end{align*}
Because $g_i \in B_k(G, T)$, $B_r(G, T) g_i^{-1} \subseteq B_{r + k}(G, T)$. On the other hand, \begin{align*}
H \cap B_{r + k}(G, T) \subseteq B_{C(r + k)+C}(H, S)=B_{Cr + D}(H, S) \subseteq \bigcup_{i=1}^m K_i h_i.
\end{align*}
Therefore \begin{align*}
B_r(G, T) \subseteq \bigcup_{i=1}^\ell (H \cap B_{r + k}(G, T)) g_i \subseteq \bigcup_{i=1}^\ell \left(\bigcup_{j=1}^m K_j h_j\right) g_i = \bigcup_{i, j}  K_j h_j g_i.
\end{align*}
The number of cosets in this cover is at most $\ell m = [G : H]\cC_{H, S}(Cr + D)$.
\end{proof}

\begin{claim}\label{lem:quotient_upper_bound}
If $Q = G/N$ is a quotient of $G$ and $\varphi \colon G \to Q$ is the quotient map, then 
\begin{align*}
  \cC_{G, S}(r) \leq \cC_{Q, \varphi(S)}(r).
\end{align*}
\end{claim}
\begin{proof}
Fix $r$, and let $m = \cC_{Q, \varphi(S)}(r)$. Say $B_r(Q, \varphi(S)) \subseteq \bigcup_{i=1}^m K_i g_i$, with each $K_i g_i$ a c.i.i.\ in $Q$. We will show that this covering can be lifted to a covering of $B_r(G, S)$ by c.i.i.s in $G$.

First note that $\varphi(B_r(G, S)) = B_r(Q, \varphi(S))$. 
It follows that 
\begin{align*}
  B_r(G, S) &\subseteq \varphi^{-1}(B_r(Q, \varphi(S))) \subseteq \varphi^{-1}\left(\bigcup_{i=1}^m K_i g_i\right) = \bigcup_{i=1}^m \varphi^{-1}(K_i g_i).
\end{align*}
Now note that each preimage $\varphi^{-1}(K_i g_i)$ corresponds to a coset of $\varphi^{-1}(K_i)$. Moreover, since each $K_i$ has infinite index in $Q$, each $\varphi^{-1}(K_i)$ has infinite index in $G$. These thus yield a covering of $B_r(G, S)$ by $m = \cC_{Q, \varphi(S)}(n)$ c.i.i.s.
\end{proof}


\subsection{Random walks on groups}
Given a finitely generated $G$, let $\mu$ a finitely supported, symmetric, non-degenerate probability measure on $G$; the latter condition means that the support of $\mu$ generates $G$ as a semigroup. Denote $\min\mu=\min\{\mu(g)\,:\, \mu(g)>0\}$. Let $X_1,X_2,\ldots$ be i.i.d.\ random variables with distribution $\mu$. Let $Z_n = X_1 \cdot X_2 \cdots X_n$ be the $\mu$-random walk on $G$.

We will need the following claim, which is a simple consequence of \cite[Lemma 3.4]{lyons2005asymptotic}.
\begin{claim}
\label{clm:lyons}
Let $(G, \mu)$ be as above. Let $H$ be an infinite index subgroup of $G$. Then for any coset $H g$ and any $n$,
\begin{align*}
    \Pr{Z_n \in H g} \leq \frac{4}{\min\mu}\cdot\frac{1}{\sqrt{n}}.
\end{align*}
\end{claim}
\begin{proof}
  Lyons~\cite{lyons2005asymptotic} considers Markov chains on a countable state space $X$, that are reversible with respect to an infinite, positive measure $\pi$. Denote the transition matrix of such a chain by  $Q(\cdot,\cdot)$, and let $c = \inf\{\pi(x)Q(x,y)\,:\, x \neq y \text{ and } Q(x,y) > 0\}$. He shows that if such a chain starts at $x \in X$, then  the probability that it is at any particular $y \in X$ at time $n$ can be bounded from above by $\frac{4\pi(x)}{c\sqrt{n}}$ (See Lemma 3.4 and  remarks 3 and 4).
  
  The process $(H Z_n)_n$ is well known to be a Markov chain on the coset space $H \backslash G$. This Markov chain is furthermore easily seen to be reversible and to have an infinite uniform stationary measure $\pi(x)=1$; the infinitude of this measure is due to $H$ having infinite index in $G$. The transition probabilities are at least $\min\mu$ whenever they are non-zero, and thus this chain satisfies the conditions of the above mentioned result, and we can conclude that the probability that this Markov chain occupies any particular state at time $n$ is at most $\frac{4}{\min\mu}\frac{1}{\sqrt{n}}$.
\end{proof}

\section{Proofs}
We start with a proof of Theorem~\ref{prop:quant-neumann}. The idea of the proof is to consider a simple random walk on the group, and to show that if the set of c.i.i.s is too small then the random walk misses it with positive probability.
\begin{proof}[Proof of Theorem~\ref{prop:quant-neumann}]
Let $H_1g_1,H_2g_2,\ldots,H_N g_N$ be cosets of infinite index of $G$.

Let $\mu$ be the uniform distribution on $S$, so that $\min\mu = 1/|S|$. Let $(Z_n)_n$ be the $\mu$-random walk. It follows from Claim~\ref{clm:lyons} that for every $i$, 
\begin{align*}
    \Pr{Z_n \in H_i g_i} \leq \frac{4|S|}{\sqrt{n}}.
\end{align*}
By the union bound
\begin{align*}
    \Pr{Z_n \in \cup_i H_i g_i} \leq \frac{4|S| N}{\sqrt{n}}.
\end{align*}

It follows that if $N < \frac{\sqrt{n}}{4|S|}$ then with positive probability $Z_n$ is not in $\cup_i H_i g_i$. Since $Z_n \in B_n$ with probability 1, it follows that at least $\frac{\sqrt{r}}{4|S|}$ c.i.i.s are needed to cover the ball of radius $r$.

\end{proof}


Our next result uses the same idea to provide sharper lower bounds on $\cC_{G, S}$ for groups in which one can upper bound the distance of a random walk from the origin. Given a monotone increasing function $f\colon \N \to \N$, we write
\[f^*(n)=\max\{k \colon f(k)\leq n\}.\]

\begin{theorem}\label{t-speed}
Let $(Z_n)$ be a finitely supported, symmetric, non-degenerate $\mu$-random walk on $G$. Denote $S = \supp\mu$. Assume that $f\colon \N \to \N$ is a function such that 
\[\mathbb{P}[Z_n\in B_{f(n)}(G, S) ]>c\]
for some $c>0$. Then 
\[\cC_{G, S}(r) \geq \frac{c\min\mu }{4}\sqrt{f^*(r)}.\]

\end{theorem}
\begin{proof}
Let $H_1g_1,H_2g_2,\ldots,H_N g_N$ be cosets of infinite index of $G$, and suppose that they cover the ball $B_{f(n)}(G, S)$, with $N=\cC_{G, S}(f(n))$. Then the event $\{Z_n\in B_{f(n)}(G, S)\}$ is contained in $\bigcup_{i=1}^N \{Z_n\in H_ig_i\}$. On the other hand, it follows from Claim~\ref{clm:lyons} that for every $i$, 
\begin{align*}
    \Pr{Z_n \in H_i g_i} \leq \frac{4}{\min\mu\sqrt{n}}.
\end{align*}
Thus, by the union bound
\begin{align*}
  c\leq  \Pr{Z_n \in  B_{f(n)}(G, S)}\leq \frac{4 N}{\min\mu\sqrt{n}}.
\end{align*}

It follows that $\cC_{G, S}(f(n))=N > \frac{c\min\mu}{4}\sqrt{n}$.
Set $n=f^*(r)$. Then $f(n)\leq r$ and since $\cC_{G, S}$ is monotone increasing we obtain
\[\cC_{G, S}(r) \geq \cC_{G, S}(f(n))\geq \frac{c\min\mu}{4}\sqrt{f^*(r)},\]
as desired. 
\end{proof}
\begin{remark}. \label{r-amenable}
The lower bound provided by Theorem \ref{t-speed} is  interesting (i.e. sharper than the general $\sqrt n$ bound in  Theorem \ref{prop:quant-neumann}) only if the function $f(n)$ satisifies $f(n)/n \to 0$ along a subsequence. A group satisfying this assumption is necessarily amenable, as can be seen by a standard application of Kesten's criterion.
\end{remark}

Recall that the speed of the random walk $(Z_n)$ is the function 
\[L_S(n)=\E{|Z_n|_S},\]
where $|\cdot|_S$ is the word norm on $(G, S)$. Then Theorem \ref{t-speed} readily implies the following.
\begin{corollary} \label{c-speed}
We have
\[\cC_{G, S}(r)\ge \frac{L_S^*(\lfloor r/2 \rfloor)^{\frac{1}{2}}}{8|S|}.\]
\end{corollary}

\begin{proof}
 Let $\mu$ be the uniform distribution on $S$. By the Markov inequality, we have 
\[ \Pr{Z_n\in B_{2L_S(n)}(G, S)}\ge \frac{1}{2}.\]
Thus we may apply Theorem \ref{t-speed} to the function $f(n)=2L_S(n)$, with $c=1/2$. Since  $f^*(n)
\ge L(\lfloor \frac{n}{2} \rfloor)$, the inequality  in the statement follows.
\end{proof}
\begin{remark}
While Corollary \ref{c-speed} is a useful criterion to apply Theorem \ref{t-speed}, \ we point out that it is sometimes a-priori easier to find  a function $f(n)$ satisfying the assumption of Theorem \ref{t-speed} directly. We do not know whether such a function is necessarily bounded below (up to constants) by  the speed of the random walk. 

\end{remark}

\subsection{Groups that admit a cautious random walk}
Theorem \ref{t-speed} provides a criterion in terms of the random walk,  under which the coset covering function must admit a linear lower bound: this holds true as long as the assumption of the theorem is satisfied for a function $f(n)\sim \sqrt{n}$. Perhaps surprisingly, a tightly related condition also implies a matching linear upper bound, thanks to results of Erschler and Ozawa \cite{erschler2018finite-dimensional} and Shalom \cite{shalom2004harmonic}.
Following Erschler and Zheng \cite[\S 6.3]{erschler2021drift}, we give the following definition. 
\begin{definition}
Let $G$ be a finitely generated group and  $\mu$  a symmetric, non-degenerate probability measure $\mu$ on $G$ with finite support $S$. We say that the random walk $(Z_n)$ on $(G, \mu)$  is \emph{cautious} if for every $\varepsilon>0$ we have
\[\inf_n \Pr{Z_n\in B_{\varepsilon \sqrt{n}}(G, S)}>0.\]
We will further say that a finitely generated group $G$ \emph{admits a cautious random walk} if there exists $\mu$ with the above properties such that the random walk on $(G, \mu)$ is cautious.
\end{definition}

It follows from the results of Hebisch and Saloff-Coste \cite{hebisch1993gaussian}  that if $G$ has polynomial growth then it admits a cautious random walk. Erschler and Zheng prove in \cite[Lemma 4.5]{erschler2020isoperimetric} that the simple random walk is cautious for every group $G$ which admits controlled F\o lner pairs in the sense of Tessera \cite{tessera2011asymptotic}. We shall not need the definition of controlled F\o lner pairs, but we mention that by results of Tessera \cite{tessera2011asymptotic,tessera2013isoperimetric} the class of groups admitting controlled F\o lner pairs includes all polycyclic groups, Baumslag-Solitar groups, lamplighter groups $F \wr \Z$ where $F$ is a finite group, solvable groups of finite Pr\"ufer rank and various lattices in solvable algebraic $p$-adic groups.

Erschler and Ozawa  show that a group that admits a cautious random walk has Shalom's property $H_{FD}$ \cite[Corollary 2.5]{erschler2018finite-dimensional}. Shalom shows that if $G$ is amenable and has property $H_{FD}$ then it virtually surjects to $\Z$ \cite[Theorem 4.3.1]{shalom2004harmonic}. Since groups that admit a cautious random walk are amenable (see Remark~\ref{r-amenable}), these results implies the following.
\begin{theorem}[Erschler and Ozawa, Shalom] Let $G$ be a finitely generated infinite group which admits a cautious random walk.
Then $G$ has a finite index subgroup which surjects onto $\Z$.
\end{theorem}

By Claims \ref{lem:finite_index_bounds}-\ref{lem:quotient_upper_bound}, this implies that if $G$ admits a cautious simple random walk, then $\cC_{G, S}(n)\le Cn$ for some $C>0$. Combining with Theorem \ref{t-speed}, we obtain Theorem~\ref{thm:poly}. In \S\ref{sec:vn} below we provide an elementary proof of this theorem, for virtually nilpotent groups.

\subsection{Groups with property (T)}
In this section we prove Theorem~\ref{thm:t}. Without loss of generality, we can assume that the generating set $S$ contains the identity. Let $H_1g_1,H_2g_2,\ldots,H_N g_N$ be cosets of infinite index of $G$. For $i \in \{1,\ldots,N\}$, let $P \colon \ell^2(H_i \backslash G) \to \ell^2(H_i \backslash G)$ be the Markov operator given by 
\begin{align*}
    [Pf](H_i g) = \frac{1}{|S|}\sum_{s \in S}f(H_i g s). 
\end{align*}
Since $G$ has property (T) and $1\in S$, there is an $\varepsilon>0$ such that the operator norm of $P$ is at most $\ee^{-\varepsilon}$, uniformly over all infinite index subgroups $H$ (see, e.g., \cite[Lemma 12.1.9]{brown2008Calgebras}). Let $f_n \in \ell^2(H_i \backslash G)$ be given by 
\begin{align*}
    f_n(H_i g) = \Pr{H_i Z_n = H_i g},
\end{align*}
and note that $f_{n+1} = P f_n$. It follows that 
the $\ell^2$-norm of $f_n$ is at most $\ee^{-\varepsilon n}$, since $f_0 = \delta_e$ has norm 1. Thus
\begin{align*}
\Pr{Z_n \in H_i g_i} = f_n(H_i g_i) \leq \ee^{-\varepsilon n}.
\end{align*}
By the union bound
\begin{align*}
    \Pr{Z_n \in \cup_i H_i g_i} \leq N \ee^{-\varepsilon n}.
\end{align*}
It follows that if $N < \ee^{\varepsilon n}$ then with positive probability $Z_n$ is not in $\cup_i H_i g_i$, and since $Z_n \in B_n$ with probability 1, at least $\ee^{\varepsilon r}$ c.i.i.s are needed to cover the ball of radius $r$.
 
\begin{remark}
The assumption that $G$ has property (T) can be weakened, as the above proof of Theorem~\ref{thm:t} only requires the existence of a unifom spectral gap for the quasi-regular representations $G \cc \ell^2(H \backslash G)$ for $H < G$ of infinite index. This property is equivalent to property FM in the sense of Cornulier \cite{cornulier2015irreducible}, namely that every action of $G$ on a set which preserves an invariant mean has a finite orbit. Property FM is implied by property (T), but it is weaker in general.
\end{remark}

\subsection{Virtually nilpotent groups}
\label{sec:vn}
In this section we provide an elementary proof of Theorem~\ref{thm:poly}, for the case of virtually nilpotent groups. The next proposition, which is the core of the proof of this theorem, shows that the so-called ``doubling condition'' implies a linear lower bound on the coset covering function. The idea of the proof is to get an upper bound on the size of the intersection of  a c.i.i.\  with  $B_r(G, S)$.

\begin{proposition}\label{prop:covering_for_doubling_groups}
Suppose that there exists some $L > 0$ and $r_0$ such that for all $r > r_0$ we have\begin{align*}
\frac{|B_{2r}(G, S)|}{|B_r(G, S)|} \leq L.
\end{align*}
Then for all $r > r_0$, \begin{align*}
\cC_{G, S}(r) \geq \frac{r}{L}.
\end{align*}
\end{proposition}
\begin{proof}
Write the generating set as $S = \{s_1, \ldots, s_m\}$, and fix $r > r_0$.

First, notice that $|B_r(G, S) \cap  C| \leq |B_{r+1}(G, S) \cap C s_i|$ for any $r$, any c.i.i.\ $C$ (in fact, for any subset of $G$) and any generator $s_i$.
This holds because $(B_r(G, S) \cap C)s_i$ is a subset of $B_{r+1}(G, S) \cap   C s_i$.

Consider the Schreier coset graph of an infinite-index subgroup $H$ in $G$. This is a graph whose nodes are (right) cosets of $H$, which has an edge between two cosets if one can be obtained from the other by right multiplication by an element of $S$. Since the Schreier graph of $H$ in $G$ is infinite ($H$ has infinite index), it contains arbitrarily long paths. Find a path of length $r$ starting at some $C= Hg$, with its edges marked by generators $s_{i_1}, s_{i_2}, \ldots, s_{i_r}$. Denote $g_k = s_{i_1}s_{i_2}\cdots s_{i_k}$. The nodes on this path are $C, C g_1, C g_2, \ldots, C g_r$.

From the argument above, we have $|B_{r+1}(G, S) \cap C g_1| \geq |B_r(G, S) \cap C|$. Similarly, 
\begin{align*}
  |B_{r+k}(G, S) \cap C g_k| 
  \geq |B_{r+k - 1}(G, S) \cap C g_{k-1}| \geq |B_r(G, S) \cap C| 
\end{align*}
for any $k \leq r$.

Moreover $|B_{2r}(G, S) \cap C g_k| \geq  |B_{r+k}(G, S) \cap C g_k|$ for $k \leq r$, and so
\begin{align*}
    |B_{2r}(G, S) \cap C g_k| \geq |B_r(G, S) \cap C|.
\end{align*}
Taking the sum from $k=1$ to $k=r$ yields
\begin{align*}
\sum_{k=1}^r |B_{2r}(G, S) \cap C g_k| \geq r|B_r(G, S) \cap C|.
\end{align*}
Since each $C g_k$ is a distinct coset of $H$, they are all disjoint. So, \begin{align*}
\sum_{k=1}^r |B_{2r}(G, S) \cap C g_k| = \left|B_{2r}(G, S) \cap \bigcup_{i=1}^r C g_k\right| \leq |B_{2r}(G, S)|.
\end{align*}
Therefore, \begin{align*}
|B_r(G, S) \cap C| \leq \frac{1}{r}|B_{2r}(G, S)| = \frac{1}{r}\frac{|B_{2r}(G, S)|}{|B_r(G, S)|}|B_r(G, S)|.
\end{align*}
Applying the doubling property now yields
\begin{align*}
     |B_r(G, S) \cap C| \leq \frac{L}{r}|B_r(G, S)|.
\end{align*}
We have thus shown that each c.i.i.\ takes at most a $\frac{L}{r}$-fraction of the ball of radius $r$. Hence at least $\frac{r}{L}$ c.i.i.s are needed to cover the ball.

\end{proof}

Given this proposition, the proof of Theorem~\ref{thm:poly} for virtually nilpotent groups is straightforward.

It is a well-known consequence of the result of Bass \cite{bass1972nilpotentgrowth} and Guivarc'h \cite{guivarch1973nilpotentgrowth} that every virtually nilpotent group satisfies the doubling condition in the hypothesis of Proposition~\ref{prop:covering_for_doubling_groups}, and thus for each such group $(G,S)$ there is a constant $c > 0$ such that $\cC_{G,S}(r) \geq c r$. For the other direction, note that every infinite nilpotent group has a quotient to $\Z$, and thus, by Claim~\ref{lem:quotient_upper_bound} there is another constant $c'$ such that $\cC_{G,S}(r) \leq c' r$. Claim~\ref{lem:finite_index_bounds} now implies that the same holds for every infinite virtually nilpotent group. This completes the proof.

\bibliography{refs}

\begin{bibdiv}
\begin{biblist}

\bib{bass1972nilpotentgrowth}{article}{
      author={{Bass}, H.},
       title={The degree of polynomial growth of finitely generated nilpotent
  groups},
        date={1972},
     journal={Proc. London Math Soc.},
      volume={s3--25},
       pages={603\ndash 614},
}

\bib{brown2008Calgebras}{book}{
      author={Brown, Nathanial~P.},
      author={Ozawa, Narutaka},
       title={{$C^*$}-algebras and finite-dimensional approximations},
      series={Graduate Studies in Mathematics},
   publisher={American Mathematical Society, Providence, RI},
        date={2008},
      volume={88},
        ISBN={978-0-8218-4381-9; 0-8218-4381-8},
         url={https://doi.org/10.1090/gsm/088},
      review={\MR{2391387}},
}

\bib{cornulier2015irreducible}{article}{
      author={Cornulier, Yves},
       title={Irreducible lattices, invariant means, and commensurating
  actions},
        date={2015},
        ISSN={0025-5874},
     journal={Math. Z.},
      volume={279},
      number={1-2},
       pages={1\ndash 26},
         url={https://doi.org/10.1007/s00209-014-1355-x},
      review={\MR{3299841}},
}

\bib{erschler2018finite-dimensional}{article}{
      author={Erschler, Anna},
      author={Ozawa, Narutaka},
       title={Finite-dimensional representations constructed from random
  walks},
        date={2018},
        ISSN={0010-2571},
     journal={Comment. Math. Helv.},
      volume={93},
      number={3},
       pages={555\ndash 586},
         url={https://doi.org/10.4171/CMH/444},
      review={\MR{3854902}},
}

\bib{erschler2020isoperimetric}{article}{
      author={Erschler, Anna},
      author={Zheng, Tianyi},
       title={Isoperimetric inequalities, shapes of {F}\o lner sets and groups
  with {S}halom's property {$H_\mathrm{FD}$}},
        date={2020},
        ISSN={0373-0956},
     journal={Ann. Inst. Fourier (Grenoble)},
      volume={70},
      number={4},
       pages={1363\ndash 1402},
}

\bib{erschler2021drift}{article}{
      author={Erschler, Anna},
      author={Zheng, Tianyi},
       title={Law of large numbers for the drift of the two-dimensional wreath
  product},
        date={2021},
     journal={Probab. Theory Related Fields},
        note={Published online},
}

\bib{guivarch1973nilpotentgrowth}{article}{
      author={Guivarc'h, Y.},
       title={Croissance polynomiale et p\'{e}riodes des fonctions
  harmoniques},
        date={1973},
     journal={Bull. Soc. Math. France},
      volume={101},
       pages={333\ndash 379},
}

\bib{hebisch1993gaussian}{article}{
      author={Hebisch, W.},
      author={Saloff-Coste, L.},
       title={Gaussian estimates for {M}arkov chains and random walks on
  groups},
        date={1993},
        ISSN={0091-1798},
     journal={Ann. Probab.},
      volume={21},
      number={2},
       pages={673\ndash 709},
      review={\MR{1217561}},
}

\bib{lyons2005asymptotic}{article}{
      author={Lyons, Russell},
       title={Asymptotic enumeration of spanning trees},
        date={2005},
     journal={Combinatorics, Probability and Computing},
      volume={14},
      number={4},
       pages={491\ndash 522},
}

\bib{neumann1954groups}{article}{
      author={Neumann, Bernhard~H},
       title={Groups covered by permutable subsets},
        date={1954},
     journal={Journal of the London Mathematical Society},
      volume={1},
      number={2},
       pages={236\ndash 248},
}

\bib{shalom2004harmonic}{article}{
      author={Shalom, Yehuda},
       title={Harmonic analysis, cohomology, and the large-scale geometry of
  amenable groups},
        date={2004},
        ISSN={0001-5962},
     journal={Acta Math.},
      volume={192},
      number={2},
       pages={119\ndash 185},
         url={https://doi.org/10.1007/BF02392739},
      review={\MR{2096453}},
}

\bib{tessera2011asymptotic}{article}{
      author={Tessera, Romain},
       title={Asymptotic isoperimetry on groups and uniform embeddings into
  {B}anach spaces},
        date={2011},
        ISSN={0010-2571},
     journal={Comment. Math. Helv.},
      volume={86},
      number={3},
       pages={499\ndash 535},
         url={https://doi.org/10.4171/CMH/232},
      review={\MR{2803851}},
}

\bib{tessera2013isoperimetric}{article}{
      author={Tessera, Romain},
       title={Isoperimetric profile and random walks on locally compact
  solvable groups},
        date={2013},
        ISSN={0213-2230},
     journal={Rev. Mat. Iberoam.},
      volume={29},
      number={2},
       pages={715\ndash 737},
      review={\MR{3047434}},
}

\end{biblist}
\end{bibdiv}
\end{document}